\documentclass[12pt]{article}
\usepackage{amsmath,amsthm,amsfonts}
\usepackage{times}
\usepackage{url}

\newtheorem{thm}{Theorem}

\newtheorem{lem}[thm]{Lemma}
\newtheorem{cor}[thm]{Corollary}
\theoremstyle{remark}
\newtheorem{rem}[thm]{Remark}

\DeclareMathOperator{\wt}{wt}
\DeclareMathOperator{\supp}{supp}
\newcommand{\FF}{\mathbb{F}}
\newcommand{\cB}{\mathcal{B}}
\newcommand{\cD}{\mathcal{D}}

\newcommand{\cP}{\mathcal{P}}

\begin{document}

\title{Self-dual codes and the non-existence of a quasi-symmetric $2$-$(37,9,8)$ design
with intersection numbers $1$ and $3$}

\author{
Masaaki Harada\thanks{
Research Center for Pure and Applied Mathematics, 
Graduate School of Information Sciences, 
Tohoku University, Sendai 980--8579, Japan},
Akihiro Munemasa\thanks{
Research Center for Pure and Applied Mathematics, 
Graduate School of Information Sciences, 
Tohoku University, Sendai 980--8579, Japan.},
 and
Vladimir D. Tonchev\thanks{
Department of Mathematical Sciences,
Michigan Technological University
Houghton, MI 49931, USA.}
}

\maketitle

\begin{abstract}
We prove that a certain binary linear code associated with the incidence matrix
of a quasi-symmetric $2$-$(37,9,8)$ design 
with intersection numbers $1$ and $3$
must be contained in 
an extremal doubly even self-dual code of length $40$.
Using the classification of extremal doubly even self-dual codes 
of length $40$, we show that
a quasi-symmetric $2$-$(37,9,8)$ design
with intersection numbers $1$ and $3$ does not exist.
\end{abstract}

%%%%%%%%%%%%%%%%%%%%%%%%%%%%%%%%%%%%%%%%%%%%%%%%%%%
\section{Introduction}

A {\em $t$-$(v,k,\lambda)$ design} $\cD$ is a pair of a set $\cP$ of $v$ points
and a collection $\cB$ of $k$-element subsets of 
$\cP$ (called blocks)
such that every $t$-element subset of $\cP$ is contained in exactly
$\lambda$ blocks.
The number of blocks that contain a given
point is traditionally denoted by $r$, and the total number of
blocks is $b$.
We assume that $t < k < v$, throughout this paper.
% A design with no repeated block is called {\em simple}.
% In this paper, we consider only simple designs.
Often a $t$-$(v,k,\lambda)$ design is simply called a $t$-design.
A $t$-design can be represented by its (block by point)
{\em incidence matrix} $A=(a_{ij})$, where $a_{ij}=1$ if the $j$th
point is contained in the $i$th block
and $a_{ij}=0$ otherwise.

A $2$-design is called {\em quasi-symmetric} if 
the number
of points in the intersection of any two distinct blocks 
takes only two values $x$ and $y$ $(x < y)$.
The values $x$ and $y$ are called the {\em  intersection numbers}.
Quasi-symmetric $2$-designs have been widely studied
because of their connections with strongly regular graphs
and binary self-orthogonal codes (see~\cite{SS}, \cite{T98}, \cite{T07}).
Four classes of quasi-symmetric $2$-designs were given 
in~\cite{N}, namely, multiples of symmetric $2$-designs,
Steiner systems with $v > k^2$, strongly resolvable $2$-designs,
and residual designs of biplanes.
A quasi-symmetric $2$-design $\cD$ is called {\em exceptional} if 
neither $\cD$ nor its complementary design belongs to one of these four classes.
%From now on, we assume that all quasi-symmetric $2$-designs are
%exceptional.
A fundamental problem  is to determine whether there is
an exceptional quasi-symmetric $2$-design for a given set of parameters.
The parameters of exceptional quasi-symmetric $2$-$(v,k,\lambda)$ designs 
with $2k \le v \le 70$
are listed in~\cite[Table~48.25]{QSHandbook}.

According to~\cite{QSHandbook},
the existence of quasi-symmetric $2$-$(37,9,8)$ design with
intersection numbers $1$ and $3$ is an open question.
Bouyuklieva and Varbanov~\cite{BV} proved that there is
no quasi-symmetric $2$-$(37,9,8)$ design 
with intersection numbers $1$ and $3$
having an automorphism of order $5$.

The existence of a quasi-symmetric $2$-$(37,9,8)$ design
with intersection numbers $1$ and $3$ has been of interest for several
reasons.
The block graph of such a design, where two blocks are adjacent
if they share three points, is a strongly regular graph
with parameters $(148, 84, 50, 44)$, and the existence of
a graph with these parameters is still unknown (see~\cite{AEB}, 
\cite[Table~3]{PSN}).

In~\cite{JT91} Jungnickel and Tonchev 
formulated as a conjecture a sufficient condition under which
a quasi-symmetric
$2$-$(v,k,s\lambda)$ design with
\[ k(k-1)=\lambda(v-1) \]
is necessarily the union of $s$ identical copies
of a symmetric $2$-$(v,k,\lambda)$ design. This conjecture, which was later
proved by Ionin and Shrikhande~\cite{IS}, and Sane~\cite{S},
left the parameters $2$-$(37,9,8)$ ($s=4$), as the smallest open case
for which a quasi-symmetric design which is not a multiple of
a symmetric design, might exist.
We note that the union of four identical copies of a symmetric $2$-$(37,9,2)$ design
%(up to isomorphism, there exist exactly four such designs~\cite{SM}),
(up to isomorphism, there exist exactly four such designs~\cite[page 37]{MR-Handbook}),
is a quasi-symmetric $2$-$(37,9,8)$ design with intersection numbers $2$ and $9$.

The parameters $2$-$(37,9,8)$ coincide with those of a derived design
with respect to a block of a (yet unknown~\cite[page 54]{MR-Handbook})
symmetric $2$-$(149,37,9)$ design. If a quasi-symmetric $2$-$(37,9,8)$
design with intersection numbers $1$ and $3$  can be embedded as a derived
design in a symmetric $2$-$(149,37,9)$ design, then the corresponding
residual design will be a quasi-symmetric $2$-$(112,28,9)$ design
with intersection numbers $6$ and $8$.  We note that according 
to~\cite[page 54]{MR-Handbook} a $2$-$(112,28,9)$ design is also unknown.

Self-dual codes, and more generally, self-orthogonal codes, have been used 
successfully for 
the construction and classification of quasi-symmetric designs, as well as for
deriving necessary conditions for their existence~\cite{Cal}, 
\cite{HMT}, \cite{HT}, \cite{JT92}, \cite{JT15}, \cite{MT}, \cite{RT}, 
\cite{T86}, \cite{T86b}.

The main result of this paper is the following theorem, which
was established by
using the recent classification of binary doubly even self-dual codes
of length $40$~\cite{BHM}.

\begin{thm}\label{thm}
A quasi-symmetric $2$-$(37,9,8)$ design
with intersection numbers $1$ and $3$ does not exist.
\end{thm}

As a corollary of Theorem~\ref{thm}, we have the following.

\begin{cor}
A  quasi-symmetric $2$-$(112,28,9)$ design
with intersection numbers $6$ and $8$ is not
embeddable as a residual design in a 
symmetric $2$-$(149,37,9)$ design.
\end{cor}

%%%%%%%%%%%%%%%%%%%%%%%%%%%%%%%%%%%%%%%%%%%%%%%%%%%
\section{Quasi-symmetric designs and doubly even self-dual codes}

A (binary) $[n,k]$ {\em code} $C$ is a $k$-dimensional vector subspace
of $\FF_2^n$,
where $\FF_2$ denotes the finite field of order $2$.
All codes in this paper are binary.
The parameter $n$ is called the {\em length} of $C$.
The {\em weight} $\wt(x)$ of a vector $x \in \FF_2^n$ is
the number of nonzero components of $x$.
The {\em support} $\supp(x)$ of a vector 
$x=(x_1,\ldots,x_n) \in \FF_2^n$ is $\{i \mid x_i=1\}$.
A vector of $C$ is called a {\em codeword}.
The minimum nonzero weight of all codewords in $C$ is called
the {\em minimum weight} of $C$, and an $[n,k]$ code with minimum
weight $d$ is called an $[n,k,d]$ code.
A code is called {\em doubly even} if all codewords have weight
%$\equiv 0 \pmod 4$.
divisible by $4$.

The {\em dual} code $C^{\perp}$ of a code
$C$ of length $n$ is defined as
$
C^{\perp}=
\{x \in \FF_2^n \mid x \cdot y = 0 \text{ for all } y \in C\},
$
where $x \cdot y$ is the standard inner product.
A code $C$ is called {\em self-orthogonal}
if $C \subset C^\perp$, and $C$ is called {\em self-dual}
if $C=C^\perp$. 
A doubly even code is self-orthogonal.
A doubly even self-dual code of length $n$ exists if and
only if $n \equiv 0 \pmod 8$.
The minimum weight $d(C)$ of a doubly even self-dual code $C$ of length
$n$ is bounded by $d(C)\le 4 \lfloor n/24 \rfloor +4$~\cite{MS73}.
A  doubly even self-dual
code meeting the upper bound is called {\em extremal}.
Two codes are {\em equivalent} if one can be
obtained from the other by permuting the coordinates.
A classification of doubly even self-dual codes 
was done for lengths up to $40$
(see~\cite{BHM} and the references given therein, and
the data can be found in~\cite{Data}).
For example, there are $16470$ inequivalent 
extremal doubly even self-dual codes of length $40$.

%%%%%%%%%%%%%%%%%%%%%%%%%%%%%%
Let $\cD$ be a $2$-$(v,k,\lambda)$ design.
Let $A$ be the incidence matrix of $\cD$.
Let $C_1,C_2$ and $C_3$ be the codes generated by the rows of 
the following matrices:
\[
A, \quad
\begin{bmatrix}
 & 1 \\
A& \vdots \\
 & 1
\end{bmatrix}
\]
%\begin{bmatrix}A&\allone\end{bmatrix}
%\text{ 
and 
% }
%\begin{bmatrix}A&\allone&\allone&\allone\end{bmatrix},
\begin{equation}\label{C3}
\begin{bmatrix}
 & 1&1&1 \\
A& \vdots &\vdots &\vdots \\
 & 1&1&1
\end{bmatrix},
\end{equation}
respectively.

\begin{lem}[\cite{T86b}]\label{lem:T}
The codes $C_1^\perp$ and $C_2^\perp$ have minimum weights
at least $(r+\lambda)/\lambda$ and
at least $(b+r)/r$, respectively.
\end{lem}
\begin{rem}
In~\cite[Lemma~2.3]{T86b}, it was claimed that
$C_2^\perp$ has minimum weight at least
$\min\{(r+\lambda)/\lambda,(b+r)/r\}$.
We note that $(r+\lambda)/\lambda=(v-1)/(k-1)+1$ and $(b+r)/r=v/k+1$.
\end{rem}

\begin{lem}\label{lem:d1}
Let $x$ be a nonzero vector of $C_3^\perp$.
Suppose that
\begin{equation}\label{eq:C3}
x \not \in \{
(0,\ldots,0,1,1,0),
(0,\ldots,0,1,0,1),
(0,\ldots,0,0,1,1)\}. 
\end{equation}
Then $\wt(x) \ge (b+r)/r$.
\end{lem}
\begin{proof}
Let $x=(x_1,x_2)$ be a nonzero vector of
$C_3^\perp$, where $x_1 \in \FF_2^v$ and $x_2 \in \FF_2^3$.

Suppose that $\wt(x_2)$ is even.
Then $x_1 \in C_1^\perp$.
From~\eqref{eq:C3},  $x_1$ is a nonzero vector.
Hence, 
$\wt(x) \ge \wt(x_1) \ge (r+\lambda)/\lambda
=(v-1)/(k-1)+1$ by Lemma~\ref{lem:T}. 

Suppose that $\wt(x_2)$ is odd.
Then $(x_1,1) \in C_2^\perp$.
Hence, $\wt(x) \ge \wt((x_1,1)) \ge (b+r)/r$
%=v/k+1$ 
by Lemma~\ref{lem:T}.
The result follows.
\end{proof}

Let $C$ be a doubly even code of length $n \equiv 0 \pmod 8$.
By Theorem~2.2 in~\cite{MST}, there is a doubly even
self-dual code of length $n$ containing $C$.

\begin{lem}\label{lem:d2}
Let $\cD$ be a quasi-symmetric $2$-$(v,k,\lambda)$ design
with intersection numbers $x$ and $y$.
Suppose that $v \equiv 5 \pmod 8$, $k \equiv 1 \pmod 4$ and
$x,y$ are odd.
Let $A$ be the incidence matrix of $\cD$.
Let $C$ be a doubly even self-dual code of length $v+3$ containing
the row span of the 
matrix~\eqref{C3}.
%following matrix:
%\[
%G=
%\begin{bmatrix}
% & 1&1&1 \\
%A& \vdots &\vdots &\vdots \\
% & 1&1&1
%\end{bmatrix}.
%\]
% \[\begin{bmatrix}A&\allone&\allone&\allone\end{bmatrix}.\]
Then $C$ has minimum weight at least $(b+r)/r$.
\end{lem}
\begin{proof}
From the assumption on the parameters of $\cD$, 
the code generated by the rows of the matrix~\eqref{C3}
is doubly even.
By Lemma~\ref{lem:d1}, the doubly even self-dual code
$C$ has minimum weight at least $(b+r)/r$.
\end{proof}

%%%%%%%%%%%%%%%%%%%%%%%%%%%%%%%%%%%%%%%%%%%%%%%%%%%
\section{Quasi-symmetric $2$-$(37,9,8)$ designs and
doubly even self-dual codes of length $40$} 

% For a subset $S$ of $\{1,\dots,40\}$, we denote by 
% $\allone_S$ the characteristic vector of $S$ 
% in $\FF_2^{40}$.

\begin{thm}
Let $A$ be the incidence matrix of a quasi-symmetric
$2$-$(37,9,8)$ design $\cD$ with intersection numbers $1$ and $3$.
Let $C$ be a doubly even self-dual $[40,20]$ code containing
the row span of the 
matrix~\eqref{C3}.
%following matrix:
%\[
%G=
%\begin{bmatrix}
% & 1&1&1 \\
%A& \vdots &\vdots &\vdots \\
% & 1&1&1
%\end{bmatrix}.
%%G=\begin{bmatrix}A&\allone&\allone&\allone\end{bmatrix}.
%\]
Then $C$ has minimum weight $8$, and there is no codeword of
weight $8$ whose support contains the last three coordinates.
\end{thm}
\begin{proof}
By Lemma~\ref{lem:d2}, $C$ has minimum weight $8$,
since the minimum weight of a doubly even self-dual code
of length $40$ is at most $8$.

Let $\cB$ be the block set of $\cD$.
Suppose that $x$ is a codeword of
weight $8$ whose support contains the last three coordinates.
Let $S=\supp(x) \setminus \{38,39,40\}$
and
\[n_i=|\{B\in\cB\mid |S\cap B|=i\}|\quad(0\leq i\leq 5).\]
Note that $n_{i}=0$ ($i=0,2,4$).
Moreover, if
%Then we have the system of equations~(see~\cite{T86b}):
%\begin{align*}
%n_1+n_3+n_5 &=148,\\
%n_1 +3n_3 +5n_5&=5 \cdot 36,\\
%\binom{3}{2}n_3 + \binom{5}{2}n_5 &=\binom{5}{2} 8.
%\end{align*}
%The system has the following unique solution:
%\[
%n_1 = 140, n_3 = 0, n_5 = 8.
%\]
%Since 
$n_5 >0$, then there exists a block containing $S$.
%$|S\cap B|=5$.  
This means that $C$ contains a codeword of weight $4$, 
which contradicts the already proven fact that $C$ has minimum weight $8$.
Thus $n_5=0$.
Counting the number of pairs $(T,B)$ with $T\subset S\cap B$, 
$|T|=2$ and $B \in \cB$,
we find $3n_3=80$. This is impossible since $n_3$ is an integer.
% the dual code of the code generated by
% the rows of the matrix \eqref{C3} contains a vector of weight $4$.
% This contradicts Lemma~\ref{lem:d2}.
\end{proof}

\section{The search method} 

All computer calculations in this section
were done with the help of {\sc Magma}~\cite{Magma}.
Starting from an extremal doubly even self-dual $[40,20,8]$ code $C$,
we take $T\subset\{1,2,\dots,40\}$ with $|T|=3$ such that no codeword
of weight $8$ contains $T$ in its support, and consider
the following set:
\[X=\{\supp(x)\setminus T\mid x\in C,\;\wt(x)=12,\;
\supp(x)\supset T\}.\]
If there is a quasi-symmetric $2$-$(37,9,8)$ design
with intersection numbers $1$ and $3$, 
an isomorphic copy can be found in $X$ for some choice of $C$ and $T$.
In order to determine whether such $X$ contains a design, we first tested
whether the graph $\Gamma_{ij}=(V,E)$ with
\begin{align*}
V&=\{B\in X\mid B\supset\{i,j\}\},\\
E&=\{\{B,B'\}\mid B,B'\in V,\;|B\cap B'|=3\},
\end{align*}
has a clique of size $\lambda=8$
for distinct $i,j\in\{1,2,\dots,40\}\setminus T$. 
% The graph $\Gamma_{ij}$ must have a clique of size $\lambda=8$.

There are $16470$ inequivalent 
extremal doubly even self-dual codes of length $40$~\cite{BHM}.
%By the above method, we 
We verified that
$15940$ 
%codes 
of them contain no 
triple $T$ such that $\Gamma_{ij}$ has a clique of size $8$ for
all $i,j\in\{1,2,\dots,40\}\setminus T$ with $i\neq j$.
%quasi-symmetric $2$-$(37,9,8)$ design.

For the remaining $530$ codes, we %examine 
used the following method.
Fix $i_0,j_0\in\{1,2,\dots,40\}\setminus T$ and enumerated
all cliques of size $8$ in the graph $\Gamma_{i_0,j_0}$.
For such a clique $K$ and distinct $i,j\in\{1,\dots,40\}\setminus T$,
consider the subgraph $\Gamma_{ij}^K$ of $\Gamma_{ij}$ induced by
the set of vertices
\[\{B\in X\mid B\supset\{i,j\},\;|B\cap B'|\in\{1,3\}
\text{ for all }B'\in K\}.\]
If there exists a quasi-symmetric $2$-$(37,9,8)$ design 
with intersection numbers $1$ and $3$
whose block set is contained in $X$, then the block set contains a clique $K$
of size $8$ in $\Gamma_{i_0,j_0}$, and 
\begin{equation}\label{GijK}
\Gamma_{ij}^K\text{ has a clique of size $8$ for all distinct
$i,j\in\{1,\dots,40\}\setminus T$.}
\end{equation}
We chose the initial $i_0,j_0$ in such a way that $\Gamma_{i_0,j_0}$
has minimal number of cliques of size $8$, for computational efficiency.
It turns out that, 
for each of the remaining $530$ codes and each choice of a
triple $T$, our choice of $i_0,j_0$ shows that no clique $K$
of size $8$ in $\Gamma_{i_0,j_0}$ satisfies the condition~\eqref{GijK}.
Therefore, we have established Theorem~\ref{thm}.
We note that the
%\bigskip
%{\bf (Please give the second method)}
CPU time for the entire search was approximately 300 hours.

%%%%%%%%%%%%%%%%%%%%
\bigskip
\noindent
{\bf Acknowledgment.}
This work was supported by JSPS KAKENHI Grant Number 15H03633.
Vladimir D. Tonchev acknowledges support by NSA Grant H98230-16-1-0011.

%%%%%%%%%%%%%%%%%  References  %%%%%%%%%%%%%%%%%%%%%%%%

\end{document}